\documentclass[12pt]{amsart}

\usepackage{fullpage, graphicx, amssymb}

\newtheorem{thm}{Theorem}[section]
\newtheorem{prop}[thm]{Proposition}

\newtheorem{definition}[thm]{Definition}
\theoremstyle{definition}

\title{A bijection between certain quarter plane walks and Motzkin paths}
\author{Karen Yeats}

\begin{document}
\maketitle

\section{Introduction}

\subsection{Problem}

In \cite{BoMi10} Bousquet-M\'elou and Mishna observed that the class of quarter plane walks using the steps $\{\searrow, \leftarrow, \uparrow\}$ is equinumerous with the class of Motzkin paths and that the class of quarter plane walks using the steps $\{\rightarrow, \searrow, \downarrow, \leftarrow, \nwarrow, \uparrow\}$ is equinumerous with the class of bicoloured Motzkin paths.  In both cases their proofs were generating function-based consequences of a larger theory and did not yield bijections between the classes.  

A bijection for the first case was given by Eu \cite{Eu10} in the context of standard Young tableaux with at most three rows.  The purpose of this short note is to give a bijection for the larger step set which restricts to Eu's bijection.  Note that this new bijection is \emph{not} the one sought by Mortimer and Prellberg in \cite{MoPr14} as is briefly discussed in the final section.

\subsection{Setup}

\begin{definition}
A class of \emph{quarter plane walks} is a class of walks in the lattice $\mathbb{Z}^2$ starting at $(0,0)$, remaining in the first quadrant $\{(i,j): i\geq 0, j\geq 0\}$, and using steps from a given set of vectors.  The \emph{length} of a quarter plane walk is the number of steps making up the walk.
\end{definition}

There are two particular classes of quarter plane walks of interest here, so let us reserve some notation for them.

\begin{definition}
Let $\mathcal{Y}$ be the class of quarter plane walks using steps from the set 
\[\{(1,-1), (-1,0), (0,1)\} = \{\searrow, \leftarrow, \uparrow\}.\]

Let $\mathcal{S}$ be the class of quarter plane walks using steps from the set \[\{(1,0), (1,-1), (0, -1), (-1,0), (-1, 1), (0,1)\} = \{\rightarrow, \searrow, \downarrow, \leftarrow, \nwarrow, \uparrow\}.\]
\end{definition}

\begin{definition}
A \emph{Motzkin path} is a walk using the steps $\{(1,1), (1,0), (1,-1)\}$ with the additional requirement that the walk ends on the $x$-axis.  

A \emph{bicolouring} of a walk is an assignment of one of two colours, say \emph{red} and \emph{black}, to each step of the walk.  
\end{definition}
Since there are no steps with negative $x$ coordinate it makes no difference if we think of Motzkin paths as restricted to the quarter plane or as restricted to the upper half plane.  We are interested in the following classes of Motzkin paths.

\begin{definition}
Let $\mathcal{M}$ be the class of Motzkin paths.
Let $\mathcal{M}_2$ be the class of bicoloured Motzkin paths.
\end{definition}

Finally, it will be convenient, in order to define the bijection, to \emph{mark} some steps of the bicoloured Motzkin paths.  In fact this is just a second bicolouring where the two colours are called \emph{marked} and \emph{unmarked}.  This class will be sufficiently useful to also give it a name.

\begin{definition}
Let $\mathcal{MM}_2$ be the class of bicoloured Motzkin paths where each step is additionally either \emph{marked} or \emph{unmarked}.
\end{definition}

The structure of the remainder of this note is as follows.  The first step is to define a map $\phi:\mathcal{S} \rightarrow \mathcal{MM}_2$.  This map is not onto but is one-to-one and $\phi^{-1}:\phi(\mathcal{S}) \rightarrow \mathcal{S}$ is easy to give.
The second step is to define $\psi:\mathcal{M}_2 \rightarrow \mathcal{MM}_2$ with the property that $f \circ \psi = \text{id}$ where $f:\mathcal{MM}_2 \rightarrow  \mathcal{M}_2$ is the map which forgets the marking.
The third step is to show $\phi(\mathcal{S}) = \psi(\mathcal{M}_2)$ and hence that $f \circ \phi: \mathcal{S} \rightarrow \mathcal{M}_2$ is a bijection with inverse $\phi^{-1} \circ \psi$.  

Once these maps are in place we can observe that $f \circ \phi$ restricted to $\mathcal{Y}$ gives all paths using only one colour and hence gives a bijection $\mathcal{Y} \rightarrow \mathcal{M}$ which turns out to be the bijection of Eu \cite{Eu10}.

\section{The maps}

\subsection{$\mathcal{S}$ to $\mathcal{MM}_2$}

\begin{definition}
  Given $w\in \mathcal{S}$ define $m=\phi(w)$ by the following procedure.

Read through $w$ from beginning to end, at each stage adding a new step to the path $m$ as follows:
\begin{enumerate}
\item If the current step in $w$ is $\uparrow$ then add a marked red $\rightarrow$ to the end of $m$.
\item If the current step in $w$ is $\rightarrow$ then add a marked black $\rightarrow$ to the end of $m$.
\item If the current step in $w$ is $\searrow$ then find the rightmost step in $m$ which is either a marked red $\rightarrow$ or a marked black $\searrow$
  \begin{enumerate}
    \item if the found step was a marked red $\rightarrow$ replace it with an unmarked red $\nearrow$
    \item if the found step was a marked black $\searrow$ replace it with an unmarked black $\rightarrow$
  \end{enumerate}
  In both cases add a marked red $\searrow$ to the end of $m$.
\item If the current step in $w$ is $\nwarrow$ then find the rightmost step in $m$ which is either a marked black $\rightarrow$ or a marked red $\searrow$
  \begin{enumerate}
    \item if the found step was a marked black $\rightarrow$ replace it with an unmarked black $\nearrow$
    \item if the found step was a marked red $\searrow$ replace it with an unmarked red $\rightarrow$
  \end{enumerate}
  In both cases add a marked black $\searrow$ to the end of $m$.
\item If the current step in $w$ is $\leftarrow$ then find the rightmost step in $m$ which is either a marked black $\rightarrow$ or a marked red $\searrow$
  \begin{enumerate}
    \item if the found step was a marked black $\rightarrow$ replace it with an unmarked black $\nearrow$
    \item if the found step was a marked red $\searrow$ replace it with an unmarked red $\rightarrow$
  \end{enumerate}
  In both cases add an unmarked red $\searrow$ to the end of $m$.
\item If the current step in $w$ is $\downarrow$ then find the rightmost step in $m$ which is either a marked red $\rightarrow$ or a marked black $\searrow$
  \begin{enumerate}
    \item if the found step was a marked red $\rightarrow$ replace it with an unmarked red $\nearrow$
    \item if the found step was a marked black $\searrow$ replace it with an unmarked black $\rightarrow$
  \end{enumerate}
  In both cases add an unmarked black $\searrow$ to the end of $m$.
\end{enumerate}
\end{definition}

Figure \ref{phi eg fig} gives a step by step example.
\begin{figure}
\includegraphics{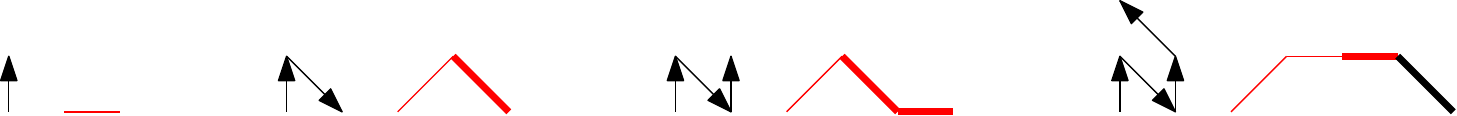}
\caption{A step by step example of $\phi$.  The bold steps of the Motzkin paths are the marked steps}\label{phi eg fig}
\end{figure}

\begin{prop}\label{endpt prop}\mbox{}
\begin{enumerate}
\item Let $m=\phi(w)$ with $w\in \mathcal{S}$.  The sum of the number of marked red $\searrow$ steps and the number of marked black $\rightarrow$ steps in $m$ is the $x$-coordinate of the end of $w$.  The sum of the number of marked red $\rightarrow$ steps and the number of marked black $\searrow$ steps in $m$ is the $y$-coordinate of the end of $w$. 
\item $\phi$ is a length-preserving map from $\mathcal{S}$ to $\mathcal{MM}_2$.
\end{enumerate}
\end{prop}

\begin{proof}
\begin{enumerate}
\item Proceed inductively.  The statement immediately holds for the two length $1$ elements of $\mathcal{S}$.  Suppose it holds for all elements of $\mathcal{S}$ of length at most $n$.  Take $w\in \mathcal{S}$ of length $n+1$ and let $w'$ be $w$ with the last step removed.  By induction the statement holds for $m'=\phi(w')$.  Consider the last step of $w$; it falls into one of the six cases in the definition of $\phi$.  

In case (1) $w$ moves one final step upwards and $m=\phi(w)$ gets one additional marked red $\rightarrow$, which contributes to the $y$-coordinate count, and so the statement holds.  Similarly in case (2) with leftwards in place of upwards and black in place of red.  In case (3) the end of $w$ moves one step right and one step down, while $m$ loses a step contributing to the $y$-coordinate count and gains one marked red $\searrow$ which contributes to the $x$-coordinate count.  Thus the statement holds for $w$.  Case (4) is analogous but reversed.  In case (5) $w$ moves one step left and $m$ loses a step contributing to the $x$-coordinate count.  Case (6) is again analogous but reversed.

In all cases the statement holds for $w$ and so by induction holds for all of $\mathcal{S}$.
\item The paths given by the construction of $\phi$ are immediately bicoloured and with each step additionally either marked or unmarked; furthermore length is preserved since each step of $w$ yields exactly one step of $\phi(w)$.  Thus we only need to check the Motzkin property and that in cases (3), (4), (5), and (6) the desired past step of $m$ is always findable.  

Again proceed inductively.  The base case is an immediate check.  Take $w\in \mathcal{S}$ and let $w'$ be $w$ with the last step removed.  By induction $m'=\phi(w')$ is a Motzkin path.  Consider the final step of $w$; it falls into one of the six cases in the definition of $\phi$.

In cases (1) and (2) the Motzkin property of $m$ holds because we are adding a $\rightarrow$ to the end of $m'$.  In cases (3) and (6) for the end of $w$ to remain in the quarter plane, the end of $w'$ must have $y$-coordinate at least $1$, and so by part (1) of this proposition the past step of $m$ needed in the construction of $\phi$ must exist.  In cases (4) and (5) the same holds with $x$ in place of $y$.  Therefore in cases (3), (4), (5), and (6), the construction of $m$ converts one step of $m'$ to end one unit higher and adds a $\searrow$ to the end of $m'$ thus preserving the Motzkin property.
Therefore $\phi:\mathcal{S} \rightarrow \mathcal{MM}_2$.
\end{enumerate}
\end{proof}

\begin{prop}\label{phi prop}
$\phi$ is one-to-one and hence $\mathcal{S}$ is in bijection with the image of $\phi$ in $\mathcal{MM}_2$.
\end{prop}

\begin{proof}
First note that the construction of $\phi$ never gives a marked $\nearrow$.

Given a bicoloured, marked Motzkin path $m$ with no marked $\nearrow$, build a lattice path $w$ using the steps $\{\rightarrow, \searrow, \downarrow, \leftarrow, \nwarrow, \uparrow\}$ as follows.  
Read through $m$ from left to right at each stage adding a new step to $w$ as follows:
\begin{enumerate}
  \item If the current step in $m$ is a marked red $\rightarrow$ or an unmarked red $\nearrow$ then add $\uparrow$ to the end of $w$.
  \item If the current step in $m$ is a marked black $\rightarrow$ or an unmarked black $\nearrow$ then add $\rightarrow$ to the end of $w$.
  \item If the current step in $m$ is a marked red $\searrow$ or an unmarked red $\rightarrow$ then add $\searrow$ to the end of $w$.
  \item If the current step in $m$ is a marked black $\searrow$ or an unmarked black $\rightarrow$ then add $\nwarrow$ to the end of $w$.
  \item If the current step in $m$ is an unmarked red $\searrow$ then add $\leftarrow$ to the end of $w$.
  \item If the current step in $m$ is an unmarked black $\searrow$ then add $\downarrow$ to the end of $w$. 
\end{enumerate}
Call the map defined above $g$.
Matching up the different possibilities we see that $g\circ \phi = \text{id}$ and hence $\phi$ is one-to-one.
\end{proof}

$\phi$ has a nice shift property.
\begin{prop}\label{shift prop}
Suppose $w\in \mathcal{S}$ can be written as the concatenation of two walks $w=w_1w_2$ with $w_1,w_2\in \mathcal{S}$.  Then $\phi(w)$ can also be written as a concatenation, $\phi(w) = \phi(w_1)\phi(w_2)$.  

Conversely suppose $\phi(w) = m_1m_2$ for some $m_1,m_2 \in \mathcal{MM}_2$.  Then $m_1= \phi(w_1)$ and $m_2 = \phi(w_2)$ for some $w_1,w_2\in \mathcal{S}$ and $w=w_1w_2$.
\end{prop}

\begin{proof}
Suppose $w=w_1w_2$ with $w_1,w_2\in \mathcal{S}$.  By the construction of $\phi$, once we have read through all the steps of $w_1$ we will have $\phi(w_1)$.  Since $\phi(w_2)$ is well-defined, it must be the case that while reading through the steps of $w_2$, in cases (3), (4), (5), and (6), the previous step can always be found among the steps of $w_2$.  Thus the steps of $w_2$ contribute $\phi(w_2)$ to the end of $\phi(w)$.  In other words $\phi(w) = \phi(w_1)\phi(w_2)$.

\medskip

Conversely suppose $\phi(w) = m_1m_2$ with $m_1, m_2\in \mathcal{MM}_2$.  Suppose $m_1$ has length $i$.  Once we have read through the first $i$ steps of $w$ in the definition of $\phi$ we will have $m_1$, because if not then some later step affects a step of $m_1$, but in such a case all subsequent steps of $m_1$ and in particular the end of $m_1$ would be lifted by one unit, and so $m_1 \not\in \mathcal{MM}_2$.  Thus $m_1 = \phi(w_1)$ where $w_1$ is the first $i$ steps of $w$.  

Let $w_2$ be the rest of $w$.  We already observed that the steps of $w_2$ do not affect $m_1$ and so $m_2 = \phi(w_2)$.  Furthermore $w_2 \in \mathcal{S}$ since otherwise at some point $w_2$ would leave the bounds of the quarter plane, at which point we would be in case (3), (4), (5), or (6), but by Proposition \ref{endpt prop} part (1) we would not be able to find an appropriate previous step in $m_2$ which is a contradiction.

Thus $\phi(w) = \phi(w_1)\phi(w_2)$ with $w_1,w_2\in \mathcal{S}$.  Then from the first part of this proposition $\phi(w) = \phi(w_1w_2)$, while by Proposition \ref{phi prop} $\phi$ is one-to-one and so $w=w_1w_2$.
\end{proof}

\subsection{$\mathcal{M}_2$ to $\mathcal{MM}_2$}

\begin{definition}
Given $s\in \mathcal{M}_2$ define $\psi(s)$ by the following procedure.
Read through $s$ from left to right marking and unmarking steps as follows.
\begin{enumerate}
  \item If the current step is a $\nearrow$ of either colour then do nothing.
  \item If the current step is a red $\rightarrow$ then find (if possible) the rightmost step before the current step which is either 
    \begin{itemize}
      \item a red $\nearrow$ which is not yet marked or unmarked or
      \item a black $\rightarrow$ which is not yet marked or unmarked.  
    \end{itemize}
    If such a step exists then set it to unmarked and do nothing with the current step.  If such a step does not exist then set the current step to marked.
  \item If the current step is a black $\rightarrow$ then find (if possible) the rightmost step before the current step which is either 
    \begin{itemize}
      \item a black $\nearrow$ which is not yet marked or unmarked or 
      \item a red $\rightarrow$ which is not yet marked or unmarked.  
    \end{itemize}
    If such a step exists then set it to unmarked and do nothing with the current step.  If such a step does not exist then set the current step to marked.
  \item If the current step is a red $\searrow$ then find (if possible) the rightmost step before the current step which is either 
    \begin{itemize}
      \item a black $\nearrow$ which is not yet marked or unmarked or 
      \item a red $\rightarrow$ which is not yet marked or unmarked.  
    \end{itemize}
    If such a step exists then set both it and the current step to unmarked.  If no such step exists then find the rightmost step before the current step which is either 
    \begin{itemize}
      \item a red $\nearrow$ which is not yet marked or unmarked or 
      \item a black $\rightarrow$ which is not yet marked or unmarked.  
    \end{itemize}
    Then set this step to unmarked and set the current step to marked.
  \item If the current step is a black $\searrow$ then find (if possible) the rightmost step before the current step which is either 
    \begin{itemize}
      \item a red $\nearrow$ which is not yet marked or unmarked or 
      \item a black $\rightarrow$ which is not yet marked or unmarked.  
    \end{itemize}
    If such a step exists then set it and the current step to unmarked.  If no such step exists then find the rightmost step before the current step which is either 
    \begin{itemize}
      \item a black $\nearrow$ which is not yet marked or unmarked or 
      \item a red $\rightarrow$ which is not yet marked or unmarked.  
    \end{itemize}
    Then set this step to unmarked and set the current step to marked.
\end{enumerate}
\end{definition}

Figure \ref{psi eg fig} gives a step by step example.
\begin{figure}
\includegraphics{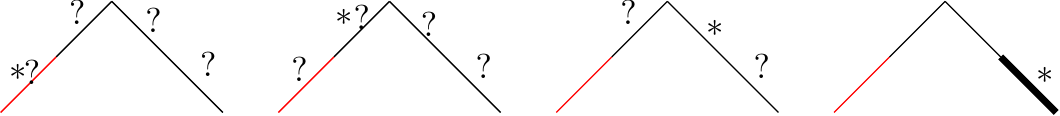}
\caption{A step by step example of $\psi$.  The bold steps are the marked steps.  The steps with ?s have unknown marking status and the step with the $*$ is the current step.}\label{psi eg fig}
\end{figure}

\begin{prop}\label{psi prop}\mbox{}
  \begin{enumerate}
  \item $\psi$ is a length-preserving map from $\mathcal{M}_2$ to $\mathcal{MM}_2$.
  \item Let $f:\mathcal{MM}_2 \rightarrow \mathcal{M}_2$ be the map which forgets the marking.  Then $f\circ \psi = \text{id}$ and hence $\mathcal{M}_2$ is in bijection with the image of $\psi$ in $\mathcal{MM}_2$.
  \end{enumerate}
\end{prop}

\begin{proof}
\begin{enumerate}
  \item Every time the construction meets a $\rightarrow$ step, one step at or before the current step which had not been previously marked or unmarked gets marked or unmarked.  Every time the construction meets a $\searrow$ step the current step and one earlier step which had not been previously marked or unmarked get marked or unmarked.  By the Motzkin property there are as many $\nearrow$ steps as $\searrow$ steps, and every $\searrow$ step has some step not yet marked or unmarked before it.  Thus the construction is always possible and at the end every step is marked or unmarked.  The remaining properties follow as $\psi$ does not change the underlying bicoloured Motzkin path.
  \item The construction of $\psi$ does not change the underlying bicoloured Motzkin path, only the markings.  Thus $f\circ \psi = \text{id}$.
\end{enumerate}  
\end{proof}

\subsection{The images}

\begin{prop}\label{image prop}
  $\phi(\mathcal{S}) = \psi(\mathcal{M}_2)$ 
\end{prop}

\begin{proof}
Let $f:\mathcal{MM}_2\rightarrow \mathcal{M}_2$ be the map which forgets the marking.

Take $w\in \mathcal{S}$.  Let $m=\phi(w)$ and $s=f(m)$.  We want to show that $\psi(s) = m$.  Read through $w$ and $s$ step by step following the definitions of $\phi$ and $\psi$ respectively.  

Claim: 
\begin{enumerate}
  \item If the current step in $s$ is a $\nearrow$ and hence has unknown marking status, then the current step in the construction of $m$ is a marked $\rightarrow$ of the same colour which will later be lifted to a $\nearrow$.  This later lifting step will be the same step which unmarks the current step in the construction of $\psi$.
  \item If the current step in $s$ is a $\rightarrow$ which has unknown marking status, then the current step in the construction of $m$ is a marked $\searrow$ of the same colour which will later be lifted to a $\rightarrow$.  This later lifting step will be the same step which unmarks the current step in the construction of $\psi$.
  \item If the current step in $s$ is a marked $\rightarrow$ then the current step in the construction of $m$ is a marked $\rightarrow$ of the same colour which will not later be raised.
  \item If the current step in $s$ is a $\searrow$ then the current step in the construction of $m$ is a $\searrow$ of the same colour and marking status. 
\end{enumerate}

Now prove the claim by induction on the steps.  At a given step we are in one of the following situations
\begin{itemize}
  \item The current step in $s$ is a $\nearrow$.  This can only come from a marked $\rightarrow$ in the partial construction of $m$ which will later be lifted to a $\nearrow$.  
  \item The current step in $s$ is a $\rightarrow$ which caused a previous step to be unmarked.  By the induction hypothesis, the conditions in the definitions of $\phi$ and $\psi$ for finding a previous step will find the same step and so this previous step is the step raised by the current step in the definition of $\phi$ and hence is also unmarked at this same point by $\phi$.
    \item The current step in $s$ is a $\rightarrow$ which has just been marked.  Then as the condition in (2) or (3) of the definition of $\psi$ was not fulfilled, by the correspondence of the induction hypothesis, there was no previous match for this step in the construction of $\phi$ and so it corresponds to a marked $\rightarrow$ in $m$.
    \item The current step in $s$ is a $\searrow$.  By the induction hypothesis, the first set of conditions in the definition of $\psi$ corresponds to the conditions to match a current unmarked $\searrow$ in the definition of $\phi$, and the matching step is raised, hence unmarked in both $\psi$ and $\phi$.  If this match is not possible then the $\searrow$ must be marked and hence raises a step using the opposite condition as given in both $\psi$ and $\phi$.
The unmarked possibility for the current step must be considered first in the definition of $\psi$ as otherwise when the step which should match with the found occurs it would match with the current step instead which would change the shape of the Motzkin path.
\end{itemize}
Using the above we can conclude that $\psi(s)=m$ and so $\phi(\mathcal{S}) \subseteq \psi(\mathcal{M}_2)$.

In the other direction take $s\in \mathcal{M}_2$.  Let $m=\psi(s)$.  Note that every $\nearrow$ of $m$ is unmarked and so we can let $w=g(m)$ where $g$ is as defined in the proof of Proposition \ref{phi prop}.  The correspondence from the claim again holds, and can again be proved inductively by reading through $w$ and $s$ step by step.  Hence $\phi(w) = m$ and so $\psi(\mathcal{M}_2) \subseteq \phi(\mathcal{S})$.

Therefore $\phi(\mathcal{S}) = \psi(\mathcal{M}_2)$.
\end{proof}

\subsection{Results}

\begin{thm}
Let $f:\mathcal{MM}_2\rightarrow \mathcal{M}_2$ be the map which forgets the marking.  Then $f\circ \phi:\mathcal{S} \rightarrow \mathcal{M}_2$ is a length-preserving bijection.
\end{thm}

\begin{proof}
Immediate from Proposition \ref{endpt prop} (2) along with Propositions \ref{phi prop}, \ref{psi prop}, and \ref{image prop}.
\end{proof}

\begin{thm}
Let $f$ be as in the previous theorem.
The image of $\mathcal{Y}$ under $f\circ \phi$ is the set of bicoloured Motzkin paths using only red steps.  Thus $f \circ \phi$ also gives a length-preserving bijection between $\mathcal{Y}$ and $\mathcal{M}$.
\end{thm}

\begin{proof}
  Consider the six cases in the construction of $\phi$.  No case changes the colour of a previous step.  Cases (1), (3), and (5) involve steps used in $\mathcal{Y}$ and add a red step to the end of the Motzkin path.  Thus the image of $\mathcal{Y}$ under $f\circ \phi$ consists of Motzkin paths using only red steps.  Cases (2), (4), and (6) involve steps not used in $\mathcal{Y}$ and add a black step to the end of the Motzkin path.  Thus any element of $\mathcal{S}\smallsetminus \mathcal{Y}$ under $f\circ \phi$ gives a Motzkin path with at least one black step.  The result follows.
\end{proof}

\section{Comments}

Eu in \cite{Eu10} gives a size-preserving bijection between the set of standard Young tableau with at most 3 rows and the set of Motzkin paths.  There is a well-known bijection between Standard Young tableau with at most 3 rows and $\mathcal{Y}$ given by the following rule:
step $i$ of the walk is
\begin{itemize}
  \item $\uparrow$ if $i$ appears in row $1$ of the Young tableau,
  \item $\searrow$ if $i$ appears in row $2$ of the Young tableau and
  \item $\leftarrow$ if $i$ appears in row $3$ of the Young tableau.
\end{itemize}
The standard Young property is then equivalent to the quarter plane condition on the walk.  Equivalently, the Yamanouchi word of the tableau is the walk using the alphabet $\{\uparrow, \searrow, \leftarrow\}$ in place of $\{1,2,3\}$.  In view of this, Eu's map associating a Motzkin path to a word (\cite{Eu10} p465) is identical to $(\phi^{-1}\circ \psi)|_{\mathcal{Y}}$;  the labelling and looking back in Eu's step A1 is the finding and marking of $\psi$.

Mortimer and Prellberg show (\cite{MoPr14}, Corollary 4), using generating function techniques, that 
\begin{itemize}
  \item walks in $\mathcal{S}$ which never go outside the triangle bounded by the axes and $x+y=2H+1$ are equinumerous with bicoloured Motzkin paths in a strip of height $H$ and
  \item walks in $\mathcal{S}$ which never go outside the triangle bounded by the axes and $x+y=2H$ are equinumerous with bicoloured Motzkin paths in a strip of height $H$ which never have a horizontal step at height $H$.
\end{itemize}
The bijection given here does not have this height property.  A natural problem then is to find a bijection which does.  After extensive playing around, I feel it is likely that there is such a bijection in which the Motzkin path is built step by step from the walk and where each new step modifies the path so far by possibly raising a step, with one small exception when $H=1$.  This would give a bijection with a similar flavour to the one described in this note.  Unfortunately the raising rule, if it exists, is obscure at present.

\appendix

\section{Examples of $\phi$}

To aid the reader in their intuition this appendix contains some examples of $\phi$.  Observe that reflecting a walk $s\in \mathcal{S}$ is equivalent to swapping the colours in $\phi(s)$, so only walks beginning $\uparrow$ will be shown.  In all the examples marked steps are fat and unmarked steps are thin.  Figure \ref{up to 3 fig} gives $\phi$ on such walks up to length 3.

\begin{figure}
\includegraphics{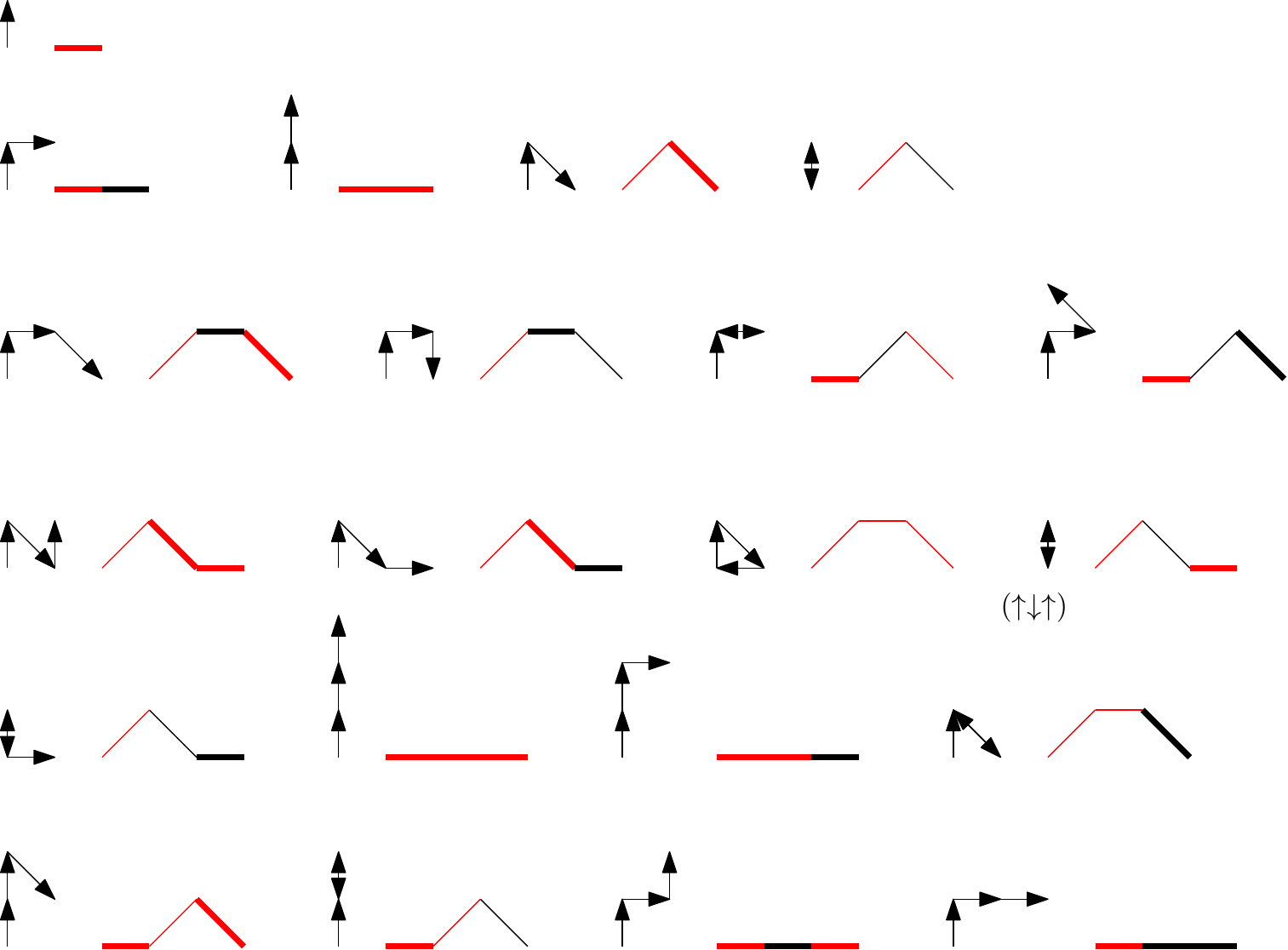}
\caption{$\phi$ on walks from $\mathcal{S}$ beginning with $\uparrow$ of length at most $3$}\label{up to 3 fig}
\end{figure}

In view of Proposition \ref{shift prop} only walks whose Motzkin paths have no interior returns to the $x$-axis  are shown for length 4, see Figure \ref{size 4 fig}.  Finally, to give a partial sense of the length 5, $\phi$ on the walks from $\mathcal{Y}$ of length 5 with Motzkin paths with no interior returns to the $x$-axis is given in Figure \ref{size 5 fig}.

\begin{figure}
\includegraphics{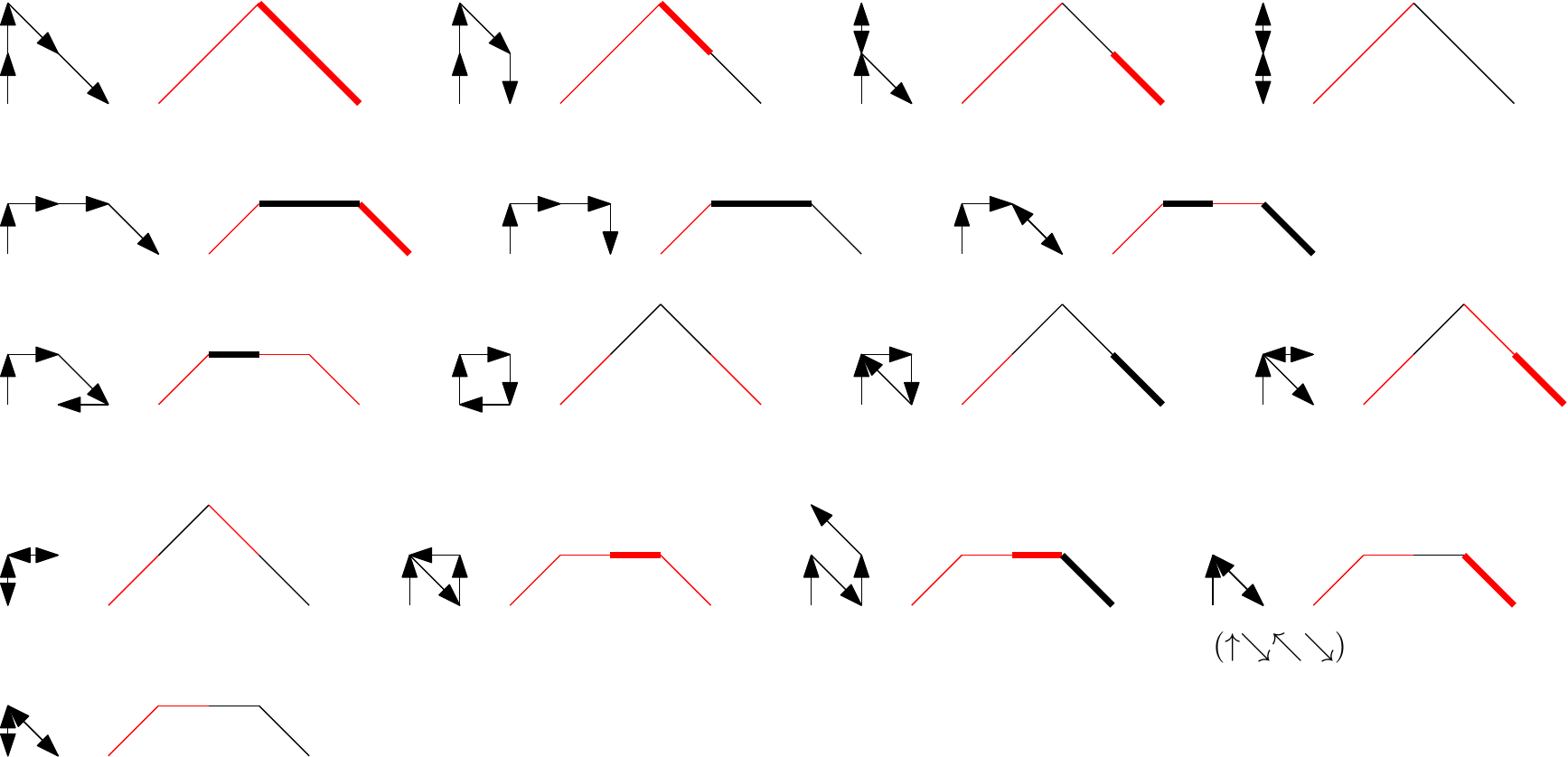}
\caption{$\phi$ on walks from $\mathcal{S}$ of length $4$ beginning with $\uparrow$ and whose Motzkin paths have no interior returns to the $x$-axis.}\label{size 4 fig}
\end{figure}

\begin{figure}
\includegraphics{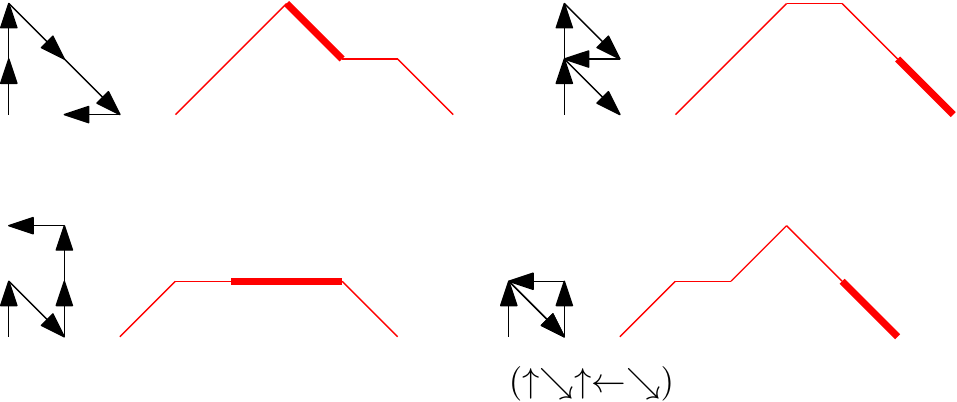}
\caption{$\phi$ on walks from $\mathcal{Y}$ of length $5$ beginning with $\uparrow$ and whose Motzkin paths have no interior returns to the $x$-axis.}\label{size 5 fig}
\end{figure}

\bibliographystyle{plain}
\bibliography{main}

\end{document}